\documentclass[10pt]{amsart}
\usepackage{amsmath,amssymb,latexsym, amsfonts, amscd, amsthm}
\usepackage{color, tabu}

\usepackage{dynkin-diagrams}
\usepackage{tikz-cd}

\usepackage{tikz}
\theoremstyle{plain}

\newtheorem{theorem}{Theorem}[section]

\newtheorem{proposition}[theorem]{Proposition}

\newtheorem{lemma}[theorem]{Lemma}

\theoremstyle{definition}

\def\bdf{\begin{defn}}
\def\edf{\end{defn}}

\usepackage{color}

\begin{document}

\title{Lifting involutions in a Weyl group to the normalizer of the torus}

\author{Moshe Adrian}
\address{Moshe Adrian:  Department of Mathematics, Queens College, CUNY, Queens, NY 11367-1597}
\email{moshe.adrian@qc.cuny.edu}

\maketitle

\begin{abstract}
Let $N$ be the normalizer of a maximal torus $T$ in a split reductive group over $\mathbb{F}_q$ and let $w$ be an involution in the Weyl group $N/T$.  We construct a section of $W$ satisfying the braid relations, such that the image of the lift $n$ of $w$ under the Frobenius map is equal to the inverse of $n$. 
\end{abstract}

\section{Introduction}\label{intro}
Let $G$ be a connected reductive algebraic group over an algebraically closed field $F$.  Let $T$ be a maximal torus in $G$, and let $W = N / T$ denote the associated Weyl group, where $N$ denotes the normalizer of $T$ in $G$, and let $X_*(T)$ denote the cocharacter lattice of $T$.  Fix a \emph{realization} of the root system $\Phi$ in $G$ (see \S\ref{prelim}), a set of positive roots $\Phi^+$, and let $\Delta$ be the associated set of simple roots.  We obtain the Tits section $w \mapsto \dot{w}$ of the natural map $N \rightarrow W$ \cite{Tit66}.   

Let us recall some setup from a recent work of Lusztig \cite{Lus18}.  If $F$ is an algebraic closure of $\mathbb{F}_q$, we define $\phi : F \rightarrow F$ by $\phi(c) = c^q$, and if $F = \mathbb{C}$, we define $\phi : F \rightarrow F$ by $\phi(c) = \overline{c}$ (complex conjugation).  In the first case, we assume that $G$ has a fixed $\mathbb{F}_q$-rational structure with Frobenius map $\phi : G \rightarrow G$ such that $\phi(t) = t^q$ for all $t \in T$.  In the second case, we assume that $G$ has a fixed $\mathbb{R}$-structure so that $G(\mathbb{R})$ is the fixed point set of an antiholomorphic involution $\phi : G \rightarrow G$ such that $\phi(y(c)) = y(\phi(c))$ for all $y \in X_*(T), c \in F^{\times}$.  We may also assume that $\phi(\dot{w}) = \dot{w}$ for any $w \in W$.

Now let $w$ be an involution in $W$.  In \cite{Lus18}, a lift $n$ of $w$ was constructed such that $\phi(n) = n^{-1}$.  The construction was quite complicated, and involved reduction arguments and case by case computations.  In this paper, we construct a natural section $\mathcal{S}$ of the entire Weyl group $W$ that satisfies the braid relations, which accomplishes the same result (namely, that $\phi(\mathcal{S}(w)) = \mathcal{S}(w)^{-1}$ for any involution $w$ in $W$).  Our methodology in one sense illustrates the power of the braid relations, allowing us to prove the main result quickly.  We moreover note that in a recent work \cite{Adr21}, all sections of the Weyl group that satisfy the braid relations were computed, for an almost-simple connected reductive group over an algebraically closed field.

We explain our method.  Let $\mathcal{S}$ be any section of $W$ (by a \emph{section} of $W$, we mean a section of the map $N \rightarrow W$).  For $\alpha \in \Delta$, we may write $\mathcal{S}(s_{\alpha}) = t_{\alpha} \dot{s}_{\alpha}$ for some $t_{\alpha} \in T$, where $s_{\alpha} \in W$ is the simple reflection associated to $\alpha$. Let $z_{\alpha} \in F^{\times}$ be arbitrary, with $\alpha \in \Delta$, and consider now the specific torus elements $t_{\alpha} = \alpha^{\vee}(z_{\alpha})$.  We show that the map $s_{\alpha} \mapsto t_{\alpha} \dot{s}_{\alpha}$ extends to a section, denoted $\mathcal{S}$, of $W$ that satisfies the braid relations.   We then prove that the condition $\phi(t_{\alpha}) = \alpha^{\vee}(-1) t_{\alpha} \ \forall \alpha \in \Delta$ implies that $\phi(\mathcal{S}(w)) = \mathcal{S}(w)^{-1}$ for any involution $w$ in $W$.  We therefore conclude that if we define a section of $W$ by the property $s_{\alpha} \mapsto \alpha^{\vee}(z_{\alpha}) \dot{s}_{\alpha}$ for all $\alpha \in \Delta$, where $z_{\alpha} \in F^{\times}$ are such that $\alpha^{\vee}(\phi(z_\alpha) z_{\alpha}^{-1}) = \alpha^{\vee}(-1)$, then we accomplish the goal of the paper.  In the finite field case, this equality can be accomplished by setting $z_{\alpha}$ to be a $q-1$ root of $-1$ for every $\alpha \in \Delta$, and in the real case the equality can be accomplished by setting $z_{\alpha}$ to be a primitive fourth root of unity for every $\alpha \in \Delta$.  Our main result, therefore, is:

\begin{theorem}\label{maintheorem}
If $F$ is an algebraic closure of $\mathbb{F}_q$, set $\zeta$ to be a $q-1$ root of $-1$.  If $F = \mathbb{C}$, set $\zeta$ to be a primitive fourth root of unity. 

For each $\alpha \in \Delta$, define the map $s_{\alpha} \mapsto \alpha^{\vee}(\zeta) \dot{s}_{\alpha}$.  This maps extends to a section $\mathcal{S} : W \rightarrow N$ that satisfies the braid relations.  Moreover, for any involution $w$ in $W$, $\phi(\mathcal{S}(w)) = \mathcal{S}(w)^{-1}$. 
\end{theorem}

\subsection{Acknowledgements}
We thank Jeffrey Adams for helpful comments on an earlier draft of this paper.

\section{Preliminaries}\label{prelim}
We first remind the reader of the definition of Tits' section from \cite{Tit66}, as well as some generalities about general sections.  We follow \cite[\S8.1, \S9.3]{Spr98} closely.  Let $G$ be a connected reductive group over an algebraically closed field $F$, let $T$ be a maximal torus in $G$, and let $\Phi$ be the associated set of roots.  For each $\alpha \in \Phi$, let $s_{\alpha}$ denote the associated reflection in the Weyl group $W = N / T$.

\begin{proposition}\cite[Proposition 8.1.1]{Spr98}\label{properties}
\begin{enumerate}
\item For $\alpha \in \Phi$ there exists an isomorphism $u_{\alpha}$ of $\mathbf{G}_a$ onto a unique closed subgroup $U_{\alpha}$ of $G$ such that $t u_{\alpha}(x) t^{-1} = u_{\alpha}(\alpha(t) x) \ (t \in T, x \in F)$.
\item $T$ and the $U_{\alpha} \ (\alpha \in \Phi)$ generate $G$.
\end{enumerate}
\end{proposition}

Tits then defines a representative $\sigma_{\alpha}$, of $s_{\alpha}$, in $N$:

\begin{lemma}\cite[Lemma 8.1.4]{Spr98}\label{realization}
\begin{enumerate}
\item The $u_{\alpha}$ may be chosen such that for all $\alpha \in \Phi$, 
\[
\sigma_{\alpha} = u_{\alpha}(1) u_{-\alpha}(-1) u_{\alpha}(1)
\]
lies in $N$ and has image $s_{\alpha}$ in $W$.  For $x \in F^{\times}$, we have
\[
u_{\alpha}(x) u_{-\alpha}(-x^{-1}) u_{\alpha}(x) = \alpha^{\vee}(x) \sigma_{\alpha};
\]
\item $\sigma_{\alpha}^2 = \alpha^{\vee}(-1)$ and $\sigma_{-\alpha} = \sigma_{\alpha}^{-1}$;
\item If $u \in U_{\alpha} - \{1 \}$ there is a unique $u' \in U_{-\alpha} - \{1 \}$ such that $u u' u \in N$;
\item If $(u_{\alpha}')_{\alpha \in \Phi}$ is a second family with the property (1) of Proposition \ref{properties} and property (1) of Lemma \ref{realization}, there exist $c_{\alpha} \in F^{\times}$ such that
\[
u_{\alpha}'(x) = u_{\alpha}(c_{\alpha} x), \ c_{\alpha} c_{-\alpha} = 1 \ (\alpha \in \Phi, x \in F).
\]
\end{enumerate}
\end{lemma}

A family $(u_{\alpha})_{\alpha \in \Phi}$ with the properties (1) of Proposition \ref{properties} and Lemma \ref{realization} is called a \emph{realization} of the root system $\Phi$ in $G$ (see \cite[\S8.1]{Spr98}).

If $\Phi^+ \subset \Phi$ is a system of positive roots and $S$ is the associated set of simple reflections, we have:

\begin{proposition}\cite[Proposition 8.3.3]{Spr98}\label{braidrelations}
Let $\mu$ be a map of $S$ into a multiplicative monoid with the property: if $s,t \in S$, $s \neq t$, then 
$$\mu(s) \mu(t) \mu(s) \cdots = \mu(t) \mu(s) \mu(t) \cdots,$$
where in both sides the number of factors is $m(s,t)$.  Then there exists a unique extension of $\mu$ to $W$ such that if $s_1 \cdots s_h$ is a reduced decomposition for $w \in W$, we have $$\mu(w) = \mu(s_1) \cdots \mu(s_h).$$
\end{proposition}

We now fix a realization $(u_{\alpha})_{\alpha \in \Phi}$ of $\Phi$ in $G$.  Let $\alpha, \beta \in \Phi$ be linearly independent.  We denote $m(\alpha, \beta)$ the order of $s_{\alpha} s_{\beta}$.  Then $m(\alpha, \beta)$ equals one of the integers $2,3,4,6$.

\begin{proposition}\label{braid}\cite[Proposition 9.3.2]{Spr98}
Assume that $\alpha$ and $\beta$ are simple roots, relative to some system of positive roots.  Then
\[
\sigma_{\alpha} \sigma_{\beta} \sigma_{\alpha} \cdots = \sigma_{\beta} \sigma_{\alpha} \sigma_{\beta} \cdots,
\]
the number of factors on either side being $m(\alpha, \beta)$.
\end{proposition}

Following \cite[\S9.3.3]{Spr98}, we fix a set of positive roots $\Phi^+ \subset \Phi$, and let $\Delta$ be the associated set of simple roots.  Let $w = s_{\alpha_1} \cdots s_{\alpha_h}$ be a reduced expression for $w \in W$, with $\alpha_1, ..., \alpha_h \in \Delta$.  The element $\mathcal{N}_{\circ}(w) := \sigma_{\alpha_1} \cdots \sigma_{\alpha_h}$ is independent of the choice of reduced expression of $w$.  We therefore obtain a section $\mathcal{N}_{\circ} : W \rightarrow N$ of the homomorphism $N \rightarrow W$.  This is the section of Tits \cite{Tit66}. 

\section{The section $\mathcal{S}$}\label{sections}

By Proposition \ref{braidrelations}, any section $\mathcal{S}$ of $W$ satisfying the braid relations is determined by its values on a set of simple reflections. Let us write $\mathcal{S}(s_{\alpha}) = t_{\alpha} \sigma_{\alpha}$ for some $t_{\alpha} \in T$.  

Let $\alpha, \beta \in \Delta$.  In order that $t_{\alpha} \sigma_{\alpha}$ and $t_{\beta} \sigma_{\beta}$ satisfy the braid relations, it is necessary and sufficient that
\begin{equation}\label{prelimequation}
t_{\alpha} \sigma_{\alpha} t_{\beta} \sigma_{\beta} t_{\alpha} \sigma_{\alpha} \cdots = t_{\beta} \sigma_{\beta} t_{\alpha} \sigma_{\alpha} t_{\beta} \sigma_{\beta} \cdots,
\end{equation}
where in both sides the number of factors is $m(\alpha, \beta)$.  As $\sigma_{\alpha} t_{\beta} \sigma_{\alpha}^{-1} = s_{\alpha}(t_{\beta})$, and since the $\sigma$ satisfy the braid relations, \eqref{prelimequation} is equivalent to
\begin{equation}\label{mainequation}
t_{\alpha} s_{\alpha}(t_{\beta}) s_{\alpha} s_{\beta}(t_{\alpha}) \cdots = t_{\beta} s_{\beta}(t_{\alpha}) s_{\beta} s_{\alpha}(t_{\beta}) \cdots.
\end{equation}

\begin{proposition}\label{braid}
Choose any $z_\alpha \in F^{\times}$, for $\alpha \in \Delta$.  Then the map $s_{\alpha} \mapsto \alpha^{\vee}(z_\alpha) \sigma_{\alpha}$ extends to a section $\mathcal{S} : W \rightarrow N$ which satisfies the braid relations.
\end{proposition}

\begin{proof}
As in the proof of \cite[Proposition 9.3.2]{Spr98}, we need only check the cases $A_1 \times A_1, A_2, B_2$, and $G_2$.  

$A_1 \times A_1$: We compute $$\alpha^{\vee}(z_{\alpha}) s_{\alpha}(\beta^{\vee}(z_{\beta})) = \alpha^{\vee}(z_{\alpha}) \beta^{\vee}(z_{\beta}) = \beta^{\vee}(z_{\beta}) s_{\beta}(\alpha^{\vee}(z_{\alpha}))$$ since $s_\alpha, s_\beta$ commute.

$A_2$: We compute 
\[
t_{\alpha} s_{\alpha}(t_{\beta}) s_{\alpha} s_{\beta} (t_{\alpha})
\]
\[
= \alpha^{\vee}(z_{\alpha}) s_{\alpha}(\beta^{\vee}(z_{\beta})) s_{\alpha} s_{\beta}(\alpha^{\vee}(z_{\alpha})) = \alpha^{\vee}(z_{\alpha}) (\alpha + \beta)^{\vee}(z_{\beta}) s_{\alpha} (\alpha + \beta)^{\vee}(z_{\alpha})
\]
\[
= \alpha^{\vee}(z_{\alpha}) (\alpha + \beta)^{\vee}(z_{\beta})  \beta^{\vee}(z_{\alpha}) =  \alpha^{\vee}(z_{\alpha}) \alpha^{\vee}(z_{\beta})  \beta^{\vee}(z_{\alpha}) \beta^{\vee}(z_{\beta})
\]
since $s_{\alpha}(\beta^{\vee}) = s_{\beta}(\alpha^{\vee}) = (\alpha + \beta)^{\vee} = \alpha^{\vee} + \beta^{\vee}$.  One computes the same result for $t_{\beta} s_{\beta}(t_{\alpha}) s_{\beta} s_{\alpha} (t_{\beta})$.

$B_2$:
Let $\alpha$ be the short root, $\beta$ the long root.  We have $s_{\alpha}(\beta^{\vee}) = \alpha^{\vee} + \beta^{\vee} = (2 \alpha + \beta)^{\vee}$ and $s_{\beta}(\alpha^{\vee}) = \alpha^{\vee} + 2 \beta^{\vee} = (\alpha + \beta)^{\vee}$.  Thus, 

\[
t_{\alpha} s_{\alpha}(t_{\beta}) s_{\alpha} s_{\beta} (t_{\alpha}) s_{\alpha} s_{\beta} s_{\alpha} (t_{\beta})
\]
\[
= \alpha^{\vee}(z_{\alpha}) s_{\alpha}(\beta^{\vee}(z_{\beta})) s_{\alpha} s_{\beta}(\alpha^{\vee}(z_{\alpha})) s_{\alpha} s_{\beta} s_{\alpha}(\beta^{\vee}(z_{\beta})) 
\]
\[
= \alpha^{\vee}(z_{\alpha}) (2 \alpha + \beta)^{\vee}(z_{\beta}) s_{\alpha} ((\alpha + \beta)^{\vee} (z_{\alpha})) s_{\alpha} s_{\beta} ((2 \alpha + \beta)^{\vee}(z_{\beta})) 
\]
\[
= \alpha^{\vee}(z_{\alpha}) (2 \alpha + \beta)^{\vee}(z_{\beta}) (\alpha + \beta)^{\vee} (z_{\alpha}) \beta^{\vee}(z_{\beta}).
\]

One may now compute the same result for $t_{\beta} s_{\beta}(t_{\alpha}) s_{\beta} s_{\alpha} (t_{\beta}) s_{\beta} s_{\alpha} s_{\beta} (t_{\alpha})$.

$G_2$:
Let $\alpha$ be the short root, $\beta$ the long root.  We have $s_{\alpha}(\beta^{\vee}) = \alpha^{\vee} + \beta^{\vee} = (3 \alpha + \beta)^{\vee}$ and $s_{\beta}(\alpha^{\vee}) = \alpha^{\vee} + 3 \beta^{\vee} = (\alpha + \beta)^{\vee}.$  Thus, 

\[
t_{\alpha} s_{\alpha}(t_{\beta}) s_{\alpha} s_{\beta} (t_{\alpha}) s_{\alpha} s_{\beta} s_{\alpha} (t_{\beta}) s_{\alpha} s_{\beta} s_{\alpha} s_{\beta}(t_{\alpha}) s_{\alpha} s_{\beta} s_{\alpha} s_{\beta} s_{\alpha}(t_{\beta})
\]
\[
= \alpha^{\vee}(z_{\alpha}) s_{\alpha}(\beta^{\vee}(z_{\beta})) s_{\alpha} s_{\beta}(\alpha^{\vee}(z_{\alpha})) s_{\alpha} s_{\beta} s_{\alpha}(\beta^{\vee}(z_{\beta})) s_{\alpha} s_{\beta} s_{\alpha} s_{\beta}(\alpha^{\vee}(z_{\alpha})) s_{\alpha} s_{\beta} s_{\alpha} s_{\beta} s_{\alpha}(\beta^{\vee}(z_{\beta}))
\]
\[
= \alpha^{\vee}(z_{\alpha}) (\alpha^{\vee}+\beta^{\vee})(z_{\beta}) (2 \alpha^{\vee} + 3 \beta^{\vee})(z_{\alpha}) (\alpha^{\vee} + 2 \beta^{\vee})(z_{\beta}) (\alpha^{\vee} + 3 \beta^{\vee})(z_{\alpha}) (\beta^{\vee})(z_{\beta}).
\]

One may now compute the same result for $t_{\beta} s_{\beta}(t_{\alpha}) s_{\beta} s_{\alpha} (t_{\beta}) s_{\beta} s_{\alpha} s_{\beta} (t_{\alpha}) s_{\beta} s_{\alpha} s_{\beta} s_{\alpha}(t_{\beta}) s_{\beta} s_{\alpha} s_{\beta} s_{\alpha} s_{\beta}(t_{\alpha}).$

\end{proof}

We need the following result about involutions, see \cite[Theorem 5.4]{Deo82}.

\begin{proposition}\label{involutions}
Any involution $w \in W$ can be obtained starting from the involution $e$ by a sequence of length-increasing operations that are either multiplication of an involution by a simple reflection $s_{\alpha}$ with which it commutes, or conjugation by a simple reflection $s_{\alpha}$ with which it does not commute.
\end{proposition}

\begin{proof}
By induction on the length $\ell(w)$. If $\ell(w) = 0$, then $w = e$.  So suppose that $\ell(w) > 0$, and let $\alpha$ be a simple root such that $\ell(ws_{\alpha}) < \ell(w)$.  Distinguish the cases on whether or not $s_{\alpha}$ commutes with $w$.  If it commutes, then $w s_{\alpha}$ is an involution, and $w$ is obtained from it by commuting multiplication by $s_{\alpha}$.  If they don't commute, then we can see as follows that $\ell(s_{\alpha} w s_{\alpha}) = \ell(w) - 2$.  From $\ell(ws_{\alpha}) < \ell(w)$, there is a reduced expression for $w$ that ends with $s_{\alpha}$, and reversing it we obtain a reduced expression for $w^{-1} = w$ starting with $s_{\alpha}$, say $s_{\alpha} s_2 \cdots s_{\ell(w)}$.  Now the exchange condition says that an expression for $w s_{\alpha}$ can be obtained by striking out one of the generators in the latter reduced expression, and it cannot be the initial $s_{\alpha}$ as that would give $s_{\alpha} w$ which is supposed to differ from $w s_{\alpha}$. Now left-multiplying by $s_{\alpha}$ gives an expression for $s_{\alpha} w s_{\alpha}$ obtained by strikig out a generator in $s_2 \cdots s_{\ell(w)}$, and therefore of length $\ell(w) - 2$.  This $s_{\alpha} w s_{\alpha}$ is an involution of shorter length than $w$, from which $w$ can be obtained by conjugation by $s_{\alpha}$ with which it does not commute.
\end{proof}

\begin{proposition}\label{condition}
Choose any $z_\alpha \in F^{\times}$, for $\alpha \in \Delta$.   Suppose that $t_{\alpha} = \alpha^{\vee}(z_{\alpha})$, for $\alpha \in \Delta$.  If $\phi(t_{\alpha}) = \alpha^{\vee}(-1) t_{\alpha}$ for all $\alpha \in \Delta$, then $\phi(\mathcal{S}(w)) = \mathcal{S}(w)^{-1}$ for any involution $w$ in $W$.
\end{proposition}

\begin{proof}
We induct using Proposition \ref{involutions}.  First we let $w = s_{\alpha}$, a simple reflection.  Then $\phi(\mathcal{S}(w)) \mathcal{S}(w) = \phi(t_{\alpha} \sigma_{\alpha}) t_{\alpha} \sigma_{\alpha} = \phi(t_{\alpha}) \sigma_{\alpha} t_{\alpha} \sigma_{\alpha} = \phi(t_{\alpha}) t_{\alpha}^{-1} \alpha^{\vee}(-1) = 1$, since $\phi(\dot{w}) = \dot{w} \ \forall w \in W$ and since $\sigma_{\alpha} t_{\alpha} \sigma_{\alpha}^{-1} = t_{\alpha}^{-1}$.

Now suppose $w \in W$ is an involution satisfying $\phi(\mathcal{S}(w)) = \mathcal{S}(w)^{-1}$.  We need to show firstly that $\phi(\mathcal{S}(s_{\alpha} w)) = \mathcal{S}(s_{\alpha} w)^{-1}$ for any $\alpha \in \Delta$ such that $s_{\alpha}$ commutes with $w$ and $\ell(s_{\alpha} w) > \ell(w)$, and secondly that $\phi(\mathcal{S}(s_{\alpha} w s_{\alpha} )) = \mathcal{S}(s_{\alpha} w s_{\alpha})^{-1}$ for any $\alpha \in \Delta$ with $\ell(s_{\alpha} w s_{\alpha}) > \ell(w)$, such that $s_{\alpha}$ and $w$ do not commute.

Suppose that $\alpha \in \Delta$, $s_{\alpha}$ commutes with $w$, and $\ell(s_{\alpha} w) > \ell(w)$.   Then
\[
\phi(\mathcal{S}(s_{\alpha} w)) \mathcal{S}(s_{\alpha} w) = \phi(t_{\alpha} \sigma_{\alpha} \mathcal{S}(w)) \mathcal{S}(w s_{\alpha}) = \phi(t_{\alpha}) \sigma_{\alpha} \mathcal{S}(w)^{-1} \mathcal{S}(w) t_{\alpha} \sigma_{\alpha}
\]
\[
= \phi(t_{\alpha}) t_{\alpha}^{-1} \alpha^{\vee}(-1) = 1,
\]
where in the above we are using that $\mathcal{S}$ satisfies the braid relations and that $s_{\alpha}$ commutes with $w$.

Moreover, for any $\alpha \in \Delta$ with $\ell(s_{\alpha} w s_{\alpha}) > \ell(w)$, we also have

\[
\phi(\mathcal{S}(s_{\alpha} w s_{\alpha})) \mathcal{S}(s_{\alpha} w s_{\alpha} ) = \phi(t_{\alpha} \sigma_{\alpha} \mathcal{S}(w) t_{\alpha} \sigma_{\alpha}) t_{\alpha} \sigma_{\alpha} \mathcal{S}(w) t_{\alpha} \sigma_{\alpha}
\]
\[
= \phi(t_{\alpha}) \sigma_{\alpha} \mathcal{S}(w)^{-1} \phi(t_{\alpha}) \sigma_{\alpha} t_{\alpha} \sigma_{\alpha} \mathcal{S}(w) t_{\alpha} \sigma_{\alpha} = \phi(t_{\alpha}) \sigma_{\alpha} \mathcal{S}(w)^{-1} \phi(t_{\alpha}) t_{\alpha}^{-1} \alpha^{\vee}(-1) \mathcal{S}(w) t_{\alpha} \sigma_{\alpha}
\]
\[
= \phi(t_{\alpha}) \sigma_{\alpha} t_{\alpha} \sigma_{\alpha} = \phi(t_{\alpha}) t_{\alpha}^{-1} \alpha^{\vee}(-1) = 1,
\]
where again we are using that $\mathcal{S}$ satisfies the braid relations, and that $s_{\alpha}$ and $w$ do not commute.
\end{proof}

\begin{proposition}\label{zeta}
\begin{enumerate} 
\item Let $F$ be an algebraic closure of $\mathbb{F}_q$.  Set $\zeta$ to be a $q-1$ root of $-1$.  If $t_{\alpha} = \alpha^{\vee}(\zeta)$ for all $\alpha \in \Delta$, then $\phi(t_{\alpha}) = \alpha^{\vee}(-1) t_{\alpha}$ for all $\alpha \in \Delta$.
\item Let $F = \mathbb{C}$.  Set $\zeta$ to be a primitive fourth root of unity.  If $t_{\alpha} = \alpha^{\vee}(\zeta)$ for all $\alpha \in \Delta$, then $\phi(t_{\alpha}) = \alpha^{\vee}(-1) t_{\alpha}$ for all $\alpha \in \Delta$.
\end{enumerate}
\end{proposition}

\begin{proof}
For (1): Note that $\phi(t_{\alpha}) = t_{\alpha}^q$, so if we set $t_{\alpha} = \alpha^{\vee}(z_{\alpha})$, then the desired equality $\phi(t_{\alpha}) = \alpha^{\vee}(-1) t_{\alpha}$ is equivalent to $\alpha^{\vee}(z_{\alpha}^{q-1}) = \alpha^{\vee}(-1)$.  Setting $z_{\alpha}$ to be a $q-1$ root of $-1$ gives us our result.

For (2): Recall that $\phi(y(c)) = y(\phi(c)) \ \forall y \in X_*(T), c \in F^{\times}$.  Setting $t_{\alpha} = \alpha^{\vee}(z_{\alpha})$, then the desired equality $\phi(t_{\alpha}) = \alpha^{\vee}(-1) t_{\alpha}$ is equivalent to $\alpha^{\vee}(\overline{z_{\alpha}} z_{\alpha}^{-1}) = \alpha^{\vee}(-1)$.  Setting $z_{\alpha}$ to be a primitive fourth root of unity gives us our result.
\end{proof}

By Proposition \ref{braid}, Proposition \ref{condition}, and Proposition \ref{zeta}, we have now proven Theorem \ref{maintheorem}.

\end{document}